\documentclass[reqno]{amsart}
\usepackage{amssymb}
\usepackage{enumerate}
\usepackage[all]{xy}

\usepackage{tikz}
\usepackage{hyperref}
\usepackage{mathrsfs}

\usepackage{xcolor}
\definecolor{DeepPink}{RGB}{255,20,147} 
\definecolor{DeepPurple}{RGB}{160,5,220}

\newtheorem{theorem}{Theorem}[section]
\newtheorem{lemma}[theorem]{Lemma}
\newtheorem{thm}[theorem]{Theorem}
\newtheorem{coro}[theorem]{Corollary}
\newtheorem{prop}[theorem]{Proposition}
\theoremstyle{definition}
\newtheorem{remark}[theorem]{Remark}
\newtheorem{example}[theorem]{Example}

\newcommand{\R}{\mathbb{R}}
\newcommand{\Z}{\mathbb{Z}}

\newcommand{\Q}{\mathbb{Q}}
\newcommand{\angbra}[1]{\left\langle #1\right\rangle}

\newcommand{\comd}[1]{$$\xymatrix{#1}$$}
\newcommand{\inj}[0]{\ar@{^{(}->}}

\newcommand{\nbar}[1]{\overline{#1}}
\newcommand{\dcup}[0]{\displaystyle\bigcup}

\newcommand{\ph}{\varphi}

\newcommand{\toup}[1]{\stackrel{#1}{\longrightarrow}}

\newcommand{\Hom}{\operatorname{Hom}}

\newcommand{\frakZ}{\mathfrak{Z}}

\newcommand{\into}{\hookrightarrow}

\newcommand{\Spec}{\operatorname{Spec}}

\newcommand{\cl}{\operatorname{cl}}

\newcommand{\calA}{\mathcal{A}}
\newcommand{\calZ}{\mathcal{Z}}

\newcommand{\trop}{\operatorname{trop}}

\newcommand{\rec}{\operatorname{rec}}
\newcommand{\T}{\mathbb{T}}
\newcommand{\an}{\mathrm{an}}

\newcommand{\fncolon}{\colon}
\newcommand{\calX}{\mathcal{X}}
\newcommand{\calY}{\mathcal{Y}}
\newcommand{\calU}{\mathcal{U}}
\newcommand{\calV}{\mathcal{V}}
\newcommand{\frakX}{\mathfrak{X}}
\newcommand{\frakY}{\mathfrak{Y}}
\newcommand{\frakU}{\mathfrak{U}}
\newcommand{\Ray}[1]{#1_{\eta}}%$\Ray{\frakX}$ is the Raynaud fiber of the formal scheme $\frakX$.
\newcommand{\ray}[1]{\Ray{#1}}
\newcommand{\smallfan}{\Sigma}
\newcommand{\bigfan}{\Sigma'}%This used to be \newcommand{\bigfan}{\Delta}
\newcommand{\smallcpx}{\Phi}
\newcommand{\bigcpx}{\Pi}

\newcommand{\htzero}[2]{#1|_{#2\times\{0\}}}
\newcommand{\htone}[2]{#1|_{#2\times\{1\}}}
\newcommand{\hti}[3]{#1|_{#2\times\{#3\}}}%First input is the fan, the second is the #2 so that the ambient space of the fan in $#2\times\R_{\geq0}$, and the third is the height we're slicing at
\newcommand{\hticurrently}[3]{\hti{#1}{#2}{#3}}%\newcommand{\hticurrently}[3]{#1_{#3}}%First input is the fan, the second is the #2 so that the ambient space of the fan in $#2\times\R_{\geq0}$, and the third is the height we're slicing at
\newcommand{\cpxToComplete}{\smallcpx}

\newcommand{\completeCpx}{\nbar{\cpxToComplete}}

\newcommand{\halfspaceSmallFan}{\Delta}
\newcommand{\halfspaceBigFan}{\Delta'}

\newcommand{\halfspaceCompleteFan}{\nbar{\halfspaceSmallFan}}
\newcommand{\Spf}{\operatorname{Spf}}%The formal spectrum of a topological $R$-algebra.
\newcommand{\mon}{\mathrm{mon}}%The category of monoids
\newcommand{\affToricVar}[1]{U(#1)}%The affine toric variety over a field corresponding to a rational cone #1 in $N_{\R}$.

\newcommand{\ToricVar}[1]{Y(#1)}%The affine toric variety (or scheme) over a field corresponding to a (finite) rational fan #1 in $N_{\R}$
\newcommand{\toricVar}[1]{\ToricVar{#1}}
\newcommand{\toricvar}[1]{\ToricVar{#1}}
\newcommand{\affTToricVar}[1]{\calU(#1)}
\newcommand{\affTToricVarPoly}[1]{\calU_{#1}}
\newcommand{\TToricVar}[1]{\calY(#1)}%The affine $\T$-toric variety over a valuation ring corresponding to a rational cone #1 in $N_{\R}\times\R_{\geq0}$.
\newcommand{\affFormalTToricVar}[1]{\frakU_{#1}}%Formal completion of $\affTToricVar$ (really $\affTToricVarPoly$)
\newcommand{\BerkSpace}[1]{\mathscr{#1}}%The style to use for Berkovich spaces. I'm making this a command because I don't know what style I'm going to end up going for.
\begin{document}
\title{A construction of algebraizable formal models}
\author{Desmond Coles, Netanel Friedenberg}
\email{dcoles@utexas.edu, nfriedenberg@tulane.edu }
\begin{abstract}
Let $X$ be a variety over a complete nontrivially valued field $K$. We construct an algebraizable formal model for the analytification of $X$ in the case $X$ admits a closed embedding into a toric variety. By algebraizable we mean that the formal model is  given by the completion along the special fiber of a locally finite type flat scheme over the valuation ring $K^\circ$. We construct the formal model via the combinatorial theory of $\T$-toric varieties over $K^{\circ}$.
\end{abstract}
\maketitle

\section{Introduction}

Let $K$ be a complete nontrivially valued field with valuation ring $K^\circ$ and value group $\Gamma$. Let $X$ be a separated finite type $K$-scheme.
The Berkovich analytification $X^{\an}$ admits a formal model in the sense of Raynaud \cite[\S8.4, Proposition 7]{BoschBook}. The purpose of this article is to construct an \textit{algebraizable} formal model for $X$, i.e., a fomal model for $X$ given by completing an a separated, locally finite type, flat $K^\circ$-scheme, $\calX$ along its special fiber.

\begin{thm}\label{thm:ExistenceOfAlgebraizableFormalModels}
If $X$ admits a closed embedding in a normal toric variety over $K$, then $X^{\an}$ has an algebraizable formal model.
\end{thm}

The hypothesis that $X$ admits a closed embedding in a normal toric variety is satisfied for any quasiprojective $K$-scheme; see, for example, the proof of \cite[Lemma 4.3]{LimitOfTrops}. 
If $K$ is algebraically closed, then 
%W\l{}odarczyk's Embedding Theorem 
W\l{}odarczyk's embedding theorem 
\cite[Theorem A]{WlodarczykEmbedding} tells us that for any normal variety $X$ the hypothesis is satisfied if and only if any two points of $X$ have a common open affine neighborhood. From an analytic perspective this hypothesis is also natural to consider because if $X$ embeds into a toric variety then $X^{\an}$ can be realized as inverse limit of tropicalizaitons \cite{LimitsOfTrop}.

For suitable $K^\circ$-models $\calX$ of $X$, there is a natural identification of the generic fiber $\ray{\frakX}$ of the formal completion $\frakX$ of $\calX$ along its special fiber with an analytic domain in $X^{\an}$. 
We prove Theorem \ref{thm:ExistenceOfAlgebraizableFormalModels} by building such a $K^\circ$-model for the ambient toric variety such that the analytic domain is the whole analytification. This is done by extending the ambient toric variety $Y$ to the trivial model over $K^\circ$, $\calY$, and then we prove the following theorem which allows us to modify the special fiber to achieve our desired result. In the following theorem $\T$ is a split torus over $K^\circ$ with  cocharacter lattice $N$, and a $\T$-toric variety is a normal, finite type, $\T$-equivaraint, $K^\circ$-model of a toric variety over $K$.

\begin{thm}\label{thm:CompleteTheRaynaudFiber}
Let $\calY$ be a normal $\T$-toric variety with generic fiber $Y$. There is a finite separable totally ramified extension $L/K$ of valued fields, a normal $\T_{L^\circ}$-toric scheme $\nbar{\calY}$ which is locally of finite type over $L^\circ$, and a $\T_{L^\circ}$-equivariant open immersion $\calY_{L^\circ}\into\nbar{\calY}$ such that
\begin{itemize}
\item the induced map $Y_L=\calY_{L}\to\nbar{\calY}_{L}$ on generic fibers is an isomorphism of $L$-varieties and
\item the natural map $\ray{\nbar{\frakY}}\to(\nbar{\calY}_{L})^{\an}\cong(Y_{L})^{\an}$ is an isomorphism of analytic spaces
\end{itemize}
where $\nbar{\frakY}$ denotes the formal completion of $\nbar{\calY}$ along its special fiber. 
If $\Gamma$ is discrete or divisible, or, more generally, if $\calY$ admits a normal equivariant completion, then we may take $L=K$.
\end{thm}
\noindent Note that by Theorem 1.1 of \cite{GammaAdmissibleCompletions} any normal $\T$-toric variety admits an equivariant completion when $\Gamma$ is discrete or divisible.

Our approach to proving Theorem \ref{thm:CompleteTheRaynaudFiber} is combinatorial. 
In \cite{GublerSoto}, Gubler and Soto classified $\T$-toric varieties in terms of certain fans, called $\Gamma$-admissible fans, in the half-space $N_{\R}\times\R_{\geq0}$; here $N_{\R}:=N\otimes_{\Z}\R$. 
The data of a $\Gamma$-admissible fan in $N_{\R}\times\R_{\geq0}$ is equivalent to that of a rational fan $\smallfan$ in $N_{\R}$ and a $\Gamma$-rational polyhedral complex $\smallcpx$ such that the recession cone $\rec P$ of any $P\in\smallcpx$ is in $\smallfan$. This allows us to reduce Theorem \ref{thm:CompleteTheRaynaudFiber} to Theorem 1.1 of \cite{combopaper}.

The paper proceeds as follows. In \S\ref{sec:Combinatorics} we recall the classification of $\T$-toric varieties by $\Gamma$-admissible fans introduced in \cite{GublerSoto}; we also prove some technical results extending known theorems from the finite type case to the locally finite type case. In \S\ref{sec:Raynaud} we recall some generalities on constructing generic fibers of formal schemes and gluing $K$-analytic spaces. Finally, in \S\ref{sec:Proofs} we discuss how to compute generic fibers of completions of $\T$-toric varieties. We then prove
Theorems \ref{thm:ExistenceOfAlgebraizableFormalModels} and \ref{thm:CompleteTheRaynaudFiber}.

\section{Background on $\T$-toric varieties}\label{sec:Combinatorics}

In this section we review the combinatorial classification of $\T$-toric varieties, following \cite{GublerSoto}. We also discuss a slight extension of this theory from finite type schemes to locally finite types, which we will use for Theorem \ref{thm:CompleteTheRaynaudFiber}.

\subsection{$\T$-toric schemes}\label{subsec:BackgroundTToric}

Let $\T$ be a split torus over $K^\circ$ with character lattice $M$ and cocharacter lattice $N$. A \emph{$\T$-toric scheme} is an integral, separated scheme 
%$\calX$ 
$\calY$ 
flat over $K^\circ$ together with an open embedding 
%$\T_K\into\calX_K$ 
$\T_K\into\calY_K$ 
of the generic fiber of $\T$ into the generic fiber of 
%$\calX$ 
$\calY$ 
such that the action of $\T_K$ on itself by translation extends to an action of $\T$ on 
%$\calX$. 
$\calY$. 
A \emph{$\T$-toric variety} is a $\T$-toric scheme which is of finite type over $K^\circ$. Every normal $\T$-toric variety arises via a combinatorial construction, which we now review. Throughout this paper we follow the notational conventions of \cite{combopaper}. Let $A$ be any additive subgroup of $\R$ (possibly all of $\R$); set $N_{A}:=N\otimes_{\Z} A$ and consider this as a subgroup of $N_{\R}$.

Let $\sigma$ be a cone in $N_{\R}\times\R_{\geq0}$, i.e., $\sigma$ is a cone in the vector space $N_{\R}\times\R$ such that $\sigma\subseteq N_{\R}\times\R_{\geq0}$. The cone $\sigma$ is \emph{$\Gamma$-admissible} if it is pointed and can be written in the form 
$$\sigma=\{(w,t)\in N_{\R}\times\R_{\geq0}\mid \angbra{u_i,w}+\gamma_it\geq0\text{ for }i=1,\ldots,m\}$$ 
for some $u_1,\ldots,u_m\in M$ and $\gamma_1,\ldots,\gamma_m\in\Gamma$. By pointed we mean that $\sigma$ contains no lines. Because $\Gamma\neq\{0\}$ every face of a $\Gamma$-admissible cone is $\Gamma$-admissible. We say that a fan, $\Delta$, in $N_{\R}\times\R_{\geq0}$ is \emph{$\Gamma$-admissible} if all of its cones are $\Gamma$-admissible.

We can characterize $\Gamma$-admissibility by looking at the intersections $\sigma\cap(N_{\R}\times\{0\})$ and $\sigma\cap(N_{\R}\times\{1\})$. Let $\pi\fncolon N_{\R}\times\R_{\geq0}\to N_{\R}$ be the projection onto the first factor. If $\sigma$ meets $N_{\R}\times\{1\}$ then $\sigma$ is $\Gamma$-admissible if and only if $\pi(\sigma\cap(N_{\R}\times\{1\}))$ is a pointed, $\Gamma$-rational polyhedron in $N_{\R}$. Recall that a polyhedron, $P$, is \emph{$\Gamma$-rational} if it can be written:
$$P=\{w\in N_{\R}\mid \angbra{u_i,w}\geq \gamma_i\text{ for }i=1,\ldots,m\}$$ 
for some $u_1,\ldots,u_m\in M$ and $\gamma_1,\ldots,\gamma_m\in\Gamma$. By pointed we mean that $P$ contains no lines. Conversely, if $P$ is a pointed, $\Gamma$-rational polyhedron, then the closed cone 
$$c(P):=\nbar{\{(tw,t)\mid t\in\R_{\geq0}, w\in P\}}\subseteq N_{\R}\times\R_{\geq0}$$ 
over $P$ is $\Gamma$-admissible. Note that $c(P)\cap (N_{\R}\times \{1 \})=P\times \{1\}$; this gives a correspondence between $\Gamma$-admissible cones that meet $N_{\R}\times \{1\}$ and $\Gamma$-rational polyhedra in $N_{\R}$. 
If $\sigma$ does not meet $N_{\R}\times\{1\}$ then $\sigma\subseteq N_{\R}\times\{0\}$ and $\sigma$ is $\Gamma$-admissible if and only if $\pi(\sigma)$ is a pointed rational cone in $N_{\R}$ (recall that rational means generated by elements of $N$). Let $\Delta$ be a fan in $N_{\R}\times \R_{\geq 0}$. For $i=0,1$ let $\hticurrently{\Delta}{N_{\R}}{i}$ denote the set $\{\pi( \sigma\cap (N_{\R}\times \{i\}) )\mid \sigma \in \Delta\}$. From this discussion we see that a fan $\Delta$ is $\Gamma$-admissible if and only if $\hticurrently{\Delta}{N_{\R}}{1}$ is a $\Gamma$-rational polyhedral complex and $\hticurrently{\Delta}{N_{\R}}{0}$ is a rational fan. In \cite{RecessionFan} the authors study when one can define a $\Gamma$-admissible fan given a polyhedral complex $N_{\R}\times \{1\}$.

Given a $\Gamma$-admissible cone $\sigma$ in $N_{\R}\times\R_{\geq0}$, the affine normal $\T$-toric scheme corresponding to $\sigma$ is $\affTToricVar{\sigma}:=\Spec K[M]^\sigma$ where
$$K[M]^\sigma:=\left\{\sum_{u\in M}\alpha_u\chi^u\in K[M]\;\middle|\; \text{for all }(w,t)\in\sigma\text{ and }u\in M,\,\angbra{u,w}+v(\alpha_u)t\geq0\right\}.$$ 
If $\Gamma$ is discrete then Gordan's lemma shows that $K[M]^\sigma$ is a finitely generated $K^\circ$-algebra. If $\Gamma$ is not discrete then $K[M]^\sigma$ is finitely generated as a $K^\circ$-algebra if and only if all of the vertices of $\sigma\cap N_{\R}\times\{1\}$ are in $N_{\Gamma}\times\{1\}$; see \cite[Proposition 6.9]{GublerGuideTrop}. 

If $\tau$ is a face of a $\Gamma$-admissible cone $\sigma$ then the inclusion $K[M]^\sigma\subset K[M]^\tau$ induces a $\T$-equivariant open immersion $\affTToricVar{\tau}\into\affTToricVar{\sigma}$. Given a $\Gamma$-admissible fan $\halfspaceSmallFan$ in $N_{\R}\times\R_{\geq0}$, the normal $\T$-toric scheme $\calY(\halfspaceSmallFan)$ corresponding to $\halfspaceSmallFan$ is obtained by gluing the schemes $\affTToricVar{\sigma}$ for $\sigma\in\halfspaceSmallFan$ along the open immersions $\affTToricVar{\tau}\into\affTToricVar{\sigma}$ for $\tau\leq\sigma$. The following theorem, which classifies the $\T$-toric varieties, was first shown in \cite[IV, \S 3]{ToroidalEmbeddings} in the case where $\Gamma$ is discrete and was then shown in \cite[Theorem 3]{GublerSoto} in the case where $\Gamma$ is not discrete.

\begin{thm}\label{thm:ClassifyNormalTToricVars}
If $\Gamma$ is discrete then $\halfspaceSmallFan\mapsto\calY(\halfspaceSmallFan)$ gives a bijection from the set of finite $\Gamma$-admissible fans in $N_{\R}\times\R_{\geq0}$ to the set of isomorphism classes of normal $\T$-toric varieties. If $\Gamma$ is not discrete then $\halfspaceSmallFan\mapsto\calY(\halfspaceSmallFan)$ gives a bijection from the set of finite $\Gamma$-admissible fans $\halfspaceSmallFan$ in $N_{\R}\times\R_{\geq0}$ such that all vertices of $\htone{\halfspaceSmallFan}{N_{\R}}$ are in $N_{\Gamma}$ to the set of isomorphism classes of normal $\T$-toric varieties.
\end{thm}

Recall that $\Q\Gamma\subset\R$ denotes the divisible hull of $\Gamma$. 
For any $\Gamma$-admissible fan $\halfspaceSmallFan$ in $N_{\R}\times\R_{\geq0}$, all of the vertices of $\htone{\halfspaceSmallFan}{N_{\R}}$ are in $N_{\Q\Gamma}$. In particular, if $\Gamma$ is divisible then $\halfspaceSmallFan\mapsto\calY(\halfspaceSmallFan)$ gives a bijection from the set of finite $\Gamma$-admissible fans in $N_{\R}\times\R_{\geq0}$ to the set of isomorphism classes of normal $\T$-toric varieties.

Let $\halfspaceSmallFan$ be a $\Gamma$-admissible fan in $N_{\R}\times\R_{\geq0}$. For any $\sigma\in\halfspaceSmallFan$, the generic fiber of $\affTToricVar{\sigma}$ is the affine toric variety over $K$ associated to $\sigma\cap(N_{\R}\times\{0\})$, viewed as a rational cone in $N_{\R}$. So if $\smallfan:=\htzero{\halfspaceSmallFan}{N_{\R}}$ is finite, then the generic fiber of $\TToricVar{\halfspaceSmallFan}$ is the toric variety $\ToricVar{\smallfan}$ over $K$ associated to $\smallfan$. More generally, if $\Sigma$ is not finite, then the generic fiber of $\TToricVar{\halfspaceSmallFan}$ is the toric scheme $\ToricVar{\smallfan}$ locally of finite type over $K$ associated to $\smallfan$ as in \cite[Theorem 4.1]{OdaTorusEmbeddings}.

\subsection{$\T$-toric schemes locally of finite type}\label{subsec:LocFinTypeTToric}
We now extend Theorem \ref{thm:ClassifyNormalTToricVars} to a classification of normal $\T$-toric schemes locally of finite type over $K^\circ$. 
%{\color{red}(Either here or in the introduction, find a good way to acknowledge that this isn't necessary for the proof of the main theorems, but that it is justified because it rounds out the study of these objects. Or, of course, find a way to incorporate them into some stronger version of something.)} %%%I think it is okay-ish for now not to do this. Plus, I hope to eventually put in results that will use this.
First, we need some general results about actions of group schemes over a base. 

\begin{lemma}\label{lemma:ActionOfUnivOpenGpIsOpen}
Let $S$ be a scheme and let $G$ be a group $S$-scheme which is universally open. If $\ph\fncolon G\times_{S}X\to X$ is an action of $G$ on an $S$-scheme $X$, then $\ph$ is open.
\end{lemma}\begin{proof}
This claim appears in \cite[Ch.\ 0 \S2 Remark (4)]{GIT}. In that remark there are other standing hypotheses, but those hypotheses are not needed for the relevant part of the remark. 
\end{proof}

%{\color{green}(It could be that I should change the following proposition into a lemma (in which case I should make sure to change the references to it!), but I'm not sure.)}%%%It could be that it should be changed, but there's nothing wrong with it as is.

\begin{prop}\label{prop:CoverByInvariantQCOpens}
Let $S$ be a scheme and let $G$ be a group $S$-scheme which is universally open and quasicompact over $S$. 
If $\ph\fncolon G\times_{S}X\to X$ is an action of $G$ on an $S$-scheme $X$, then every point of $X$ is contained in a $G$-invariant quasicompact open subscheme of $X$.
\end{prop}\begin{proof}
Given $x\in X$, let $U'$ be an affine open neighborhood of $x$. Then $G\times_{S} U'$ is a quasicompact scheme. Let $U$ be the set-theoretic image $\ph(G\times_S U')$. Then by Lemma \ref{lemma:ActionOfUnivOpenGpIsOpen}, $U$ is an open neighborhood of $x$. Since $G\times_S U'$ is quasicompact and $\ph$ is continuous, $U$ is quasicompact.
\end{proof}

\begin{remark}\label{remark:FlatAndFinPresHyp}
By \cite[Th\'{e}or\`{e}me 2.4.6]{EGAIV2}, the hypotheses of Proposition \ref{prop:CoverByInvariantQCOpens} are satisfied if $G$ is flat and of finite presentation over $S$.
\end{remark}

%%%%%I originally re-defined $K,v,K^\circ,\Gamma,\T,M,$ and $N$. I realized this, and it seems the options are to recall what everything is, or to just procede. I'm doing the later, but have included a version of the text for the former in a comment.
We now apply Proposition \ref{prop:CoverByInvariantQCOpens} to the case of $\T$-toric schemes locally of finite type over $K^\circ$.

\begin{prop}\label{prop:ValRingSumihiroPlus}
Any normal $\T$-toric scheme which is locally of finite type over $K^\circ$ is covered by $\T$-invariant affine open subschemes.
\end{prop}\begin{proof}
Let $\calY$ be a normal $\T$-toric scheme which is locally of finite type over $K^\circ$. Since $\T$ is flat and of finite presentation over $K^\circ$, Proposition \ref{prop:CoverByInvariantQCOpens} and Remark \ref{remark:FlatAndFinPresHyp} tell us that any point $y\in\calY$ is contained in a $\T$-invariant quasicompact open subscheme $\calU$. Then $\calU$ is a normal $\T$-toric variety and so, by \cite[Theorem 2]{GublerSoto} $y$ is contained in a $\T$-invariant affine open subscheme.
\end{proof}

\begin{remark}
Proposition \ref{prop:CoverByInvariantQCOpens} can also be used to remove the hypothesis that $X$ is quasicompact from other theorems about the existence of certain types of open covers, such as \cite[Corollary 3.11]{SumihiroCompletionII}. 
To relax the finite type hypothesis in that result to the hypothesis that $X$ is locally of finite type, the only additional fact needed is that $X$ is covered by invariant open subschemes that are quasicompact over the base scheme $S$. 
However, this is immediate if $S$ is quasiseparated, as is the case in the aforementioned result, where $S$ is assumed to be noetherian. 
\end{remark}

\begin{thm}\label{thm:ClassifyNormalTToricSchemesLocOfFinType}
If $\Gamma$ is discrete then $\halfspaceSmallFan\mapsto\calY(\halfspaceSmallFan)$ gives a bijection from the set of $\Gamma$-admissible fans in $N_{\R}\times\R_{\geq0}$ to the set of isomorphism classes of normal $\T$-toric schemes locally of finite type over $K^\circ$. If $\Gamma$ is not discrete then $\halfspaceSmallFan\mapsto\calY(\halfspaceSmallFan)$ gives a bijection from the set of $\Gamma$-admissible fans $\halfspaceSmallFan$ in $N_{\R}\times\R_{\geq0}$ such that all vertices of $\htone{\halfspaceSmallFan}{N_{\R}}$ are in $N_{\Gamma}$ to the set of isomorphism classes of normal $\T$-toric schemes locally of finite type over $K^\circ$.
\end{thm}\begin{proof}
The proof is exactly the same as the proof of \cite[Theorem 3]{GublerSoto}, except that we use Proposition \ref{prop:ValRingSumihiroPlus} rather than \cite[Theorem 2]{GublerSoto} to see that the $\T$-toric scheme in question has an open cover by affine $\T$-toric varieties.
\end{proof}

\section{Raynaud's formal models}\label{sec:Raynaud}

We briefly recall the gluing procedure for $K$-analytic spaces from \cite[\S 1.3]{Berkovich93Etale}, the construction of generic fibers of formal schemes, and the relation to Berkovich analytification.
We refer the reader to \cite{BerkovichBook} and \cite{Berkovich93Etale} for background on Berkovich analytic spaces and to \cite{TemkinIntroToBerkovichSpaces} for an introduction to the subject with many exercises. 

Suppose we have a family $(\BerkSpace{X}_i)_{i\in I}$ of $K$-affinoid spaces and for each $i,j\in I$ we have an affinoid domain $\BerkSpace{X}_{ij}\subset \BerkSpace{X}_{i}$ and an isomorphism $\phi_{ji}\fncolon\BerkSpace{X}_{ij}\to \BerkSpace{X}_{ji}$ such that $\BerkSpace{X}_{ii}=\BerkSpace{X}_i$, $\phi_{ji}(\BerkSpace{X}_{ij}\cap \BerkSpace{X}_{ik})=\BerkSpace{X}_{ji}\cap \BerkSpace{X}_{jk}$, and $\phi_{ki}=\phi_{kj}\circ\phi_{ji}$ on $\BerkSpace{X}_{ij}\cap \BerkSpace{X}_{ik}$. 
Assume that, for each $i\in I$, all but finitely many of the $\BerkSpace{X}_{ij}$s are empty. 
A \emph{gluing of the $\BerkSpace{X}_{i}$s along the $\BerkSpace{X}_{ij}$s} is a $K$-analytic space $\BerkSpace{X}$ together with maps $\phi_i\fncolon\BerkSpace{X}_i\to \BerkSpace{X}$ identifying $\BerkSpace{X}_i$ with an affinoid domain in $\BerkSpace{X}$ and satisfying
\begin{itemize}
\item $\phi_i(\BerkSpace{X}_{ij})=\phi_i(\BerkSpace{X}_i)\cap\phi_j(\BerkSpace{X}_j)$,
\item $\phi_i=\phi_j\circ\phi_{ji}$ on $\BerkSpace{X}_{ij}$, and
\item $\{\phi_i(\BerkSpace{X}_i)\mid i\in I\}$ is a quasinet on $\BerkSpace{X}$, i.e., every point $x\in \BerkSpace{X}$ has a neighborhood of the form $\bigcup_{j=1}^n\phi_{i_j}(\BerkSpace{X}_{i_j})$ with $x\in\bigcap_{j=1}^n\phi_{i_j}(\BerkSpace{X}_{i_j})$.
\end{itemize}
By \cite[Proposition 1.3.3(b)]{Berkovich93Etale}, such a gluing exists and is unique up to unique isomorphism. 
Furthermore, \cite[Proposition 1.3.2]{Berkovich93Etale} tells us that maps from $\BerkSpace{X}$ to any $K$-analytic space $\BerkSpace{Y}$ are naturally in bijection with families $(f_i)_{i\in I}$ of maps $f_i\fncolon\BerkSpace{X}_i\to \BerkSpace{Y}$ such that $f_i=f_j\circ\phi_{ji}$ on $\BerkSpace{X}_{ij}$.

In order to ensure that the above gluing process can be used to construct the generic fiber of a formal scheme as a Berkovich space we will need to impose a topological condition on our formal schemes. 
Following \cite[\S 8.2 Definition 12]{BoschBook}, we say that a topological space $X$ is \emph{quasi-paracompact} if $X$ admits a cover $\{U_i\}_{i\in I}$ consisting of quasicompact open sets such that, for each $i\in I$, $U_i\cap U_j=\emptyset$ for all but finitely many $j\in I$.

We briefly recall the definition of admissible formal $K^\circ$-schemes. 
A topological $K^\circ$-algebra $A$ is \emph{admissible} if there is some nonzero $\alpha\in K^{\circ\circ}$ such that $A$ is $\alpha$-torsion free, has the $\alpha$-adic topology, and is isomorphic as a $K^\circ$-algebra to a quotient of $K^\circ\angbra{\zeta_1,\ldots,\zeta_n}$, the $\alpha$-adic completion of the polynomial ring $K^\circ[\zeta_1,\ldots,\zeta_n]$ \cite[\S7.3, Definition 3 and Corollary 5]{BoschBook}. Each of these properties is independent of the choice of $\alpha$. A formal $K^\circ$-scheme is \emph{admissible} if it has an open cover by formal spectra of admissible $K^\circ$-algebras.

We now recall the construction of the generic fiber of a suitable formal scheme as a Berkovich analytic space. We refer the reader to \cite[\S 7.4]{BoschBook} for details and an accessible presentation of the analogous construction as a rigid space. 
If $A$ is an admissible $K^\circ$-algebra then $K\otimes_{K^\circ}A$ is a $K$-affinoid algebra. This gives rise to a functor from affine admissible formal schemes over $K^\circ$ to $K$-affinoid spaces, sending $\Spf A$ to the Berkovich spectrum $\mathscr{M}(K\otimes_{K^\circ}A)$ of $K\otimes_{K^\circ}A$. This functor sends inclusions of affine open formal subschemes to inclusions of $K$-affinoid domains. This can be extended to a functor that sends a separated, quasi-paracompact, admissible formal $K^\circ$-scheme $\frakX$ to a $K$-analytic spaces $\ray{\frakX}$, as follows. Because $\frakX$ is quasi-paracompact and admissible, there is a cover $\{\frakU_i\mid i\in I\}$ of $\frakX$ by formal spectra $\frakU_i=\Spf A_i$ of admissible $K^\circ$-algebras $A_i$ such that, for each $i\in I$, $\frakU_i\cap\frakU_j=\emptyset$ for all but finitely many $j\in I$. 
Moreover, because $\frakX$ is separated, each $\frakU_i\cap\frakU_j$ is also the formal spectrum of an admissible $K^\circ$-algebra $A_{ij}$. We obtain $\ray{\frakX}$ by gluing the $\mathscr{M}(K\otimes_{K^\circ}A_i)$s along the $\mathscr{M}(K\otimes_{K^\circ}A_{ij})$s. 
We call $\ray{\frakX}$ the \emph{generic fiber} of $\frakX$. 
Note that if $\frakX$ is an admissible formal scheme which is not quasi-paracompact, the generic fiber may not exist as a Berkovich analytic space.

Given a $K$-analytic space $\BerkSpace{Z}$, a \emph{formal model of $\BerkSpace{Z}$} consists of a separated, quasi-paracompact, admissible formal $K^\circ$-scheme $\mathfrak{Z}$ and an isomorphism $\BerkSpace{Z}\toup{\sim}\ray{\mathfrak{Z}}$.

Let $\calX$ be a scheme which is separated, flat, and locally of finite type over $K^\circ$. Suppose that the special fiber $\calX_s$ of $\calX$ is quasi-paracompact. 
Let $\frakX$ be the formal completion of $\calX$ along the special fiber, by which we mean the $\alpha$-adic completion for any nonzero $\alpha\in K^{\circ\circ}$, and let $X=\calX_K$ be the generic fiber of $\calX$. 
Note that the hypotheses above guarantee that $\frakX$ is a separated, quasi-paracompact, admissible formal $K^\circ$-scheme. There is a natural map $\iota=\iota_{\calX}\fncolon\ray{\frakX}\to X^{\an}$ from the generic fiber of $\frakX$ to the analytification of $X$, defined as follows. 
If $\calX=\Spec A$ is affine then there is a natural identification $\iota$ of $\ray{\frakX}$ with the affinoid domain 
$$\{x\in X^{\an}\mid \text{for all } f\in A,\,|f(x)|\leq 1\};$$ 
see \cite[\S 5]{BerkovichFormalSchemesI} or \cite[\S 4.13]{GublerGuideTrop}. 
Furthermore, if $\calU$ is an affine open subscheme of $\calX$ with generic fiber $U$ and formal completion $\frakU$ along its special fiber, then the diagram 
\comd{
\ray{\frakU}\ar[r]^{\iota_{\calU}}\ar[d]&U^{\an}\ar[d]\\
\ray{\frakX}\ar[r]^{\iota_\calX}&X^{\an}
}
commutes, where the vertical maps are induced by the inclusion $\calU\into\calX$. 
For an arbitrary $\calX$, the previous sentence gives us that for any open affine subschemes $\calU_1,\calU_2\subset\calX$ the maps 
$\ray{(\frakU_{i})}\toup{\iota_{\calU_i}}U_i^{\an}\to X^{\an}$ 
for $i=1,2$ agree on $\ray{(\frakU_1\cap\frakU_2)}$, so the universal property of gluing gives us the map $\iota\fncolon\ray{\frakX}\to X^{\an}$. 

If $\calU_1=\Spec A_1$ and $\calU_2=\Spec A_2$ are open affine subschemes of $\calX$ then because $\calX$ is separated we have that $\calU_{12}:=\calU_1\cap\calU_2$ is affine with coordinate ring $A_{12}$ which is generated as a $K^\circ$-algebra by the images of $A_1$ and $A_2$. 
%Note that $\iota(\ray{(\frakU_1)})\cap\iota(\ray{(\frakU_2)})$ is contained in $U_1^{\an}\cap U_2^{\an}=U_{12}^{\an}$ where it takes the form $\{x\in U_{12}^{\an}\mid (\forall f\in A_1\cup A_2)|f(x)|\leq 1\}$. 
%%%%%%Awk line break
Note that $\iota(\ray{(\frakU_1)})\cap\iota(\ray{(\frakU_2)})$ is contained in $U_1^{\an}\cap U_2^{\an}=U_{12}^{\an}$ where it takes the form 
%$$\{x\in U_{12}^{\an}\mid (\forall f\in A_1\cup A_2)|f(x)|\leq 1\}.$$ 
%%%Sam says I shouldn't use \forall
$$\{x\in U_{12}^{\an}\mid \text{for all }f\in A_1\cup A_2,\,|f(x)|\leq 1\}.$$
Since $A_1\cup A_2$ generates $A_{12}$, we get that $\iota(\ray{(\frakU_1)})\cap\iota(\ray{(\frakU_2)})=\iota(\ray{(\frakU_{12})})$.

\begin{remark}\label{remark:FlatSubschemeCompletion} Let $\calX$ have generic fiber $X$.
We will say that $\calX$ satisfies condition (*) if there is a collection $\{\calU_i\mid i\in I\}$ of affine open subsets of $\calX$ such that $\bigcup_{i\in I}\calU_i$ contains the special fiber of $\calX$, for each $i\in I$, $\frakU_i$ meets only finitely many $\frakU_j$ for $j\in I$ and $\{\iota(\ray{(\frakU_i)})\mid i\in I\}$ is a quasinet on $\iota(\ray{\frakX})\subset X^\an$. 
In this situation $\iota(\ray{\frakX})$ is an analytic domain in $X^{\an}$ and, in light of the previous paragraph, we see that $\iota(\ray{\frakX})$ is a gluing of the $\ray{(\frakU_i)}$s along the $\ray{(\frakU_i\cap\frakU_j)}$s, so $\iota$ gives an isomorphism from $\ray{\frakX}$ to $\iota(\ray{\frakX})$. If $\calX$ satisfies condition (*) with the collection $\{\calV_i\mid i\in I\}$, then for any closed subscheme $\calZ\subset\calX$ which is flat over $K^\circ$, the collection $\{\calU_i:=\calV_i\cap\calZ\mid i\in I\}$ shows that $\calZ$ satisfies condition (*). 
Moreover, in this case we have $\iota_{\calZ}(\ray{\frakZ})=Z^{\an}\cap\iota_{\calX}(\ray{\frakX})$ in $X^{\an}$. Here $Z$ is the generic fiber of $\calZ$, and $\frakZ$ is the formal completion of $\calZ$ along its special fiber. However, in the absence of condition (*) there are examples where $\iota$ is not an isomorphism on to an analytic domain in $X^{\an}$; see Example \ref{example:CombLocFinDoesntImplyRaynaudInBerkovich}. 
\end{remark}

\section{Proofs of main theorems}\label{sec:Proofs}

We now prove Theorems \ref{thm:ExistenceOfAlgebraizableFormalModels} and \ref{thm:CompleteTheRaynaudFiber}.
Let $\calY = \calY (\Delta)$ be a normal $\T$-toric scheme which is locally of finite type over $K^\circ$, given by a $\Gamma$-admissible fan $\Delta$ in $N_{\R}\times\R_{\geq0}$.
We begin by discussing how one computes the generic fiber of the completion of a  $\calY$ along its special fiber. We then give the proofs of the main theorems.

We begin by discussing some finiteness conditions on $\Delta$. Let $\smallfan:=\htzero{\halfspaceSmallFan}{N_{\R}}$ and $\smallcpx:=\htone{\halfspaceSmallFan}{N_{\R}}$. Firstly, notice that the generic fiber $\calY$ is a variety if and only if $\Sigma$ is finite, in which case the generic fiber is the toric variety given by $\Sigma$, $Y(\Sigma)$. Because we are only interested in $\calY$ with generic fiber a variety, we will assume $\Sigma$ is finite. Secondly, Theorem \ref{thm:CompleteTheRaynaudFiber} is proved by modifying the special fiber of $\calY$; to construct and compute the formal completion along the special fiber as in Lemma \ref{lemma:ClosuresLocFinInSuppImpliesRaynaudInBerkovich}, we will need to consider some finiteness conditions on $\Phi$. We say that $\smallcpx$ in $N_{\R}$ is \emph{combinatorially locally finite} if every polyhedron in $\smallcpx$ meets only finitely many other polyhedra in $\smallcpx$. An even stronger condition on $\Phi$ would be local finiteness. Let $Z$ be a topological space and $\calA $ a collection of subsets of $Z$. Then we say $\calA$ is \emph{locally fintite} if every point of $Z$ has a neighborhood that meets at most finitely elements of $\calA$. Locally finite implies combinatorially locally finite but the converse is not true as the following example shows.
\begin{example}\label{example:LocFin:CombNotTop}
Let $N=\Z$ and let $\smallcpx$ be the polyhedral complex in $N_{\R}=\R$ whose maximal cells are given by $[\frac{1}{n+1},\frac{1}{n}]$ for $n\in\Z_{>0}$ and the vertex 0. Then $\smallcpx$ is combinatorially locally finite, but is not locally finite at $0$.
\end{example}
\noindent For each $P\in\smallcpx$ set $\affTToricVarPoly{P}:=\affTToricVar{c(P)}$ where $c(P)\subset N_{\R}\times\R_{\geq0}$ is the closed cone over $P$. 
Note that, for $\sigma,\tau\in\halfspaceSmallFan$, $\affTToricVar{\sigma}\cap\affTToricVar{\tau}=\affTToricVar{\sigma\cap\tau}$ and $\affTToricVar{\sigma}$ is contained in the generic fiber of $\calY$ if and only if $\sigma$ is contained in $N_{\R}\times\{0\}$. 
So the special fiber $\calY_s$ of $\calY$ has the affine open cover $\{(\affTToricVarPoly{P})_s\mid P\in\smallcpx\}$ and, for any $P_1,P_2\in\smallcpx$, $(\affTToricVarPoly{P_1})_s$ meets $(\affTToricVarPoly{P_2})_s$ if and only if $P_1$ meets $P_2$. From this we see that when $\Phi$ is combinatorially locally finite in $N_{\R}$ the special fiber $\calY_s$ is quasi-paracompact.
Let $\frakY =\frakY(\Delta)$ be the formal completion of $\calY$ along its special fiber.
When $\Delta$ is locally combinatorially finite, we can consider the generic fiber $\ray{\frakY}$ of $\frakY$ and the natural map $\iota$ from $\ray{\frakY}$ to the analytification of $\ToricVar{\smallfan}$.

In order to study $\iota$ we briefly recall the construction of the (extended) tropicalization map, referring the reader to \cite[\S 3]{LimitOfTrops} for more detail. Denote the generic fiber of $\T$ by $T$ and its character lattice by $M$. So $Y(\Sigma)$ has dense torus $T$. For any $\sigma\in\smallfan$ there is a corresponding open affine subset $\Spec K[S_\sigma]=\affToricVar{\sigma}\subset\toricvar{\smallfan}$, where $S_{\sigma}$ is the submonoid of $M$ determined by $\sigma$.
There is also a continuous map $\trop_\sigma\fncolon\affToricVar{\sigma}^{\an}\to N_{\R}(\sigma)=\Hom_{\mon}(S_\sigma,\nbar{\R})$ defined by sending $x\in \affToricVar{\sigma}^{\an}$ to the homomorphism $S_\sigma\to\nbar{\R}$ given by $u\mapsto-\log|u(x)|$.
The maps $\trop_\sigma$ for $\sigma\in\smallfan$ glue to give a continuous map 
$$\trop\fncolon\toricVar{\smallfan}^{\an}\to N_{\R}(\smallfan),$$ 
called the \emph{tropicalization map}. 
For $\sigma\in\smallfan$ and $w\in N_{\R}(\sigma)\subset N_{\R}(\smallfan)$ we have $\trop_\sigma^{-1}(w)=\trop^{-1}(w)$. Note that $N_{\R} ( \{ 0 \} )= N_{\R}$ and $N_{\R}(\Sigma)$ is a partial compactification of $N_{ \R }$.

For $P\in\smallcpx$, let $\affFormalTToricVar{P}$ be the formal completion of $\affTToricVarPoly{P}$ along its special fiber. 
By \cite[Proposition 6.9]{RabinoffTropicalAnalyticGeometry} and \cite[Proposition 6.19]{GublerGuideTrop}, $\iota$ maps $\ray{(\affFormalTToricVar{P})}$ isomorphically onto the affinoid domain $\trop^{-1}(\cl_{N_{\R}(\smallfan)}(P))\subset\toricVar{\smallfan}^{\an}$.
Here $\cl_{N_{\R}(\smallfan)}$ denotes the closure in $N_{\R}(\smallfan)$.
In the case where $\halfspaceSmallFan$ is finite, Gubler studied $\calY$ and the generic fiber $\ray{\frakY}\cong\trop^{-1}\left(\bigcup_{P\in\smallcpx}\cl_{N_{\R}(\smallfan)}(P)\right)$ in \cite{GublerGuideTrop}. 
Generalizing this to the case where $\Phi$ is locally finite
will be very similar to the proof in the finite case. Let $|\Phi|=\cup_{P\in \Phi}P$. 
We note that we do need to assume that $\Phi$ is locally finite in Lemma \ref{lemma:ClosuresLocFinInSuppImpliesRaynaudInBerkovich} and not just combinatorially locally finite, as the following example shows.

\begin{example}\label{example:CombLocFinDoesntImplyRaynaudInBerkovich}
Let $N=\Z$ and say $\Gamma$ is $\Z$ or $\Q$. Let $\smallcpx$ be as in Example \ref{example:LocFin:CombNotTop}, and let $\halfspaceSmallFan$ be the fan in $N_{\R}\times\R_{\geq0}$ whose maximal cones are $c(P)$ for $P\in\smallcpx$ maximal. Then $\ray{\frakY}$ is the disjoint union of $\trop^{-1}(0)$ and $\trop^{-1}((0,1])$. 
Since $\trop$ is continuous, admits a continuous section \cite[\S 3.3]{TropicalSkeletons}, and has connected fibers, \cite[Ch.\ 1, \S 3.5, Proposition 9 and Ch.\ 1, \S 11.3, Proposition 7]{BourbakiGenTopI} show that $\trop^{-1}([0,1])$ is connected while $\ray{\frakY}$ has two connected components, so they are not isomorphic.
\end{example} 

\begin{lemma}\label{lemma:ClosuresLocFinInSuppImpliesRaynaudInBerkovich}
Suppose that $\calA:=\{\cl_{N_{\R}(\smallfan)}(P)\mid P\in\htone{\halfspaceSmallFan}{N_{\R}}\}$ is locally finite in $|\calA|:=\cup_{A\in\calA}A$. 
Then for any closed subscheme $\calX\subset\TToricVar{\halfspaceSmallFan}$ which is flat over $K^\circ$ with generic fiber $X$ and formal completion $\frakX$ along the special fiber, the natural map $\iota_{\calX}\fncolon\ray{\frakX}\to X^{\an}$ identifies $\ray{\frakX}$ with the analytic domain $\trop^{-1}(|\calA|)\cap X^{\an}$.
\end{lemma}\begin{proof}
Because $\iota$ maps $\ray{(\affFormalTToricVar{P})}$ onto $\trop^{-1}(\cl_{N_{\R}(\smallfan)}(P))$ for each $P\in\smallcpx$, Remark \ref{remark:FlatSubschemeCompletion} gives us that it suffices to show that $\calY$ satisfies condition (*) with the collection $\{\affTToricVarPoly{P}\mid P\in\smallcpx\}$. 
We already know that this is a collection of affine opens containing the special fiber and that each $\affFormalTToricVar{P}$ only meets finitely many $\affFormalTToricVar{P'}$ for $P'\in\smallcpx$. Thus it remains only to show that 
$$\{\iota(\ray{(\affFormalTToricVar{P})})\mid P\in\smallcpx\}=\{\trop^{-1}(\cl_{N_{\R}(\smallfan)}(P))\mid P\in\smallcpx\}$$ 
is a quasinet on 
$\iota(\ray{\frakY})=\trop^{-1}(|\calA|)\subset\toricVar{\smallfan}^{\an}$. 
Since $\{\cl_{N_{\R}(\smallfan)}(P)\mid P\in\smallcpx\}$ is a locally finite cover of $|\calA|$ by closed subsets, it is a quasinet on $|\calA|$. So because $\trop\fncolon\trop^{-1}(|\calA|)\to|\calA|$ is continuous, $\{\trop^{-1}(\cl_{N_{\R}(\smallfan)}(P))\mid P\in\smallcpx\}$ is a quasinet on 
$\trop^{-1}(|\calA|)$.
\end{proof}

\begin{coro}\label{coro:LocFinImpliesRaynaudInBerkovich}
Suppose that $\smallcpx$ is locally finite in $N_{\R}(\smallfan)$. 
Then for any closed subscheme $\calX\subset\TToricVar{\halfspaceSmallFan}$ which is flat over $K^\circ$ with generic fiber $X$ and formal completion $\frakX$ along the special fiber, the natural map $\iota_{\calX}\fncolon\ray{\frakX}\to X^{\an}$ identifies $\ray{\frakX}$ with the analytic domain $\trop^{-1}(\cl_{N_{\R}(\smallfan)}|\smallcpx|)\cap X^{\an}$.
\end{coro}\begin{proof}
Since $\smallcpx$ is locally finite in $N_{\R}(\smallfan)$, so is $\{\cl_{N_{\R}(\smallfan)}(P)\mid P\in\smallcpx\}$. So by Lemma \ref{lemma:ClosuresLocFinInSuppImpliesRaynaudInBerkovich} we only need to show that $\cl_{N_{\R}(\smallfan)}|\smallcpx|=\dcup_{P\in\smallcpx}\cl_{N_{\R}(\smallfan)}(P)$. But this follows from the fact that $\smallcpx$ is locally finite in $N_{\R}(\smallfan)$ \cite[Ch.\ 1, \S 1.5, Proposition 4]{BourbakiGenTopI}.
\end{proof}
Given the hypotheses in Lemma \ref{lemma:ClosuresLocFinInSuppImpliesRaynaudInBerkovich} and Corollary \ref{coro:LocFinImpliesRaynaudInBerkovich}, one might hope that it would be enough to have $\smallcpx$ locally finite in $|\smallcpx|$. As the following example shows, this is not sufficient.

\begin{example}
Let $N=\Z^2$ and say $\Gamma$ is $\Z$ or $\Q$. For any positive integer $n$ let $P_n:=\{(x,y)\in\R^2\mid \frac{1}{n+1}\leq x\leq\frac{1}{n}, y\geq(2n+1)-n(n+1)x\}$. That is, $P_n$ is the polyhedron in $\R^2$ with vertices $(\frac{1}{n},n)$ and $(\frac{1}{n+1},n+1)$ and recession cone $\{(0,y)\in\R^2\mid y\geq0\}$. Let $\smallfan$ be the fan whose maximal cone is $\{(0,y)\in\R^2\mid y\geq0\}$, let $P_0:=\{(0,y)\in\R^2\mid y\geq0\}$, and let $\smallcpx$ be the polyhedral complex whose maximal faces are $P_n$ for $n\geq0$. Then $\smallcpx$ is locally finite in $N_{\R}$, but not in $N_{\R}(\smallfan)$. So if we let $\halfspaceSmallFan$ be the fan in $\R^2\times\R_{\geq0}$ whose maximal cones are $c(P_n)$ for $n\geq0$, then, as in Example \ref{example:CombLocFinDoesntImplyRaynaudInBerkovich}, we find that $\ray{\frakY}$ has two connected components while $\trop^{-1}\left(\bigcup_{n\geq0}\cl_{N_{\R}(\smallfan)}(P_n)\right)$ is connected, so they are not isomorphic.
\end{example}

We now have all of the ingredients we need to prove Theorem \ref{thm:CompleteTheRaynaudFiber}.

\begin{proof}[Proof of Theorem \ref{thm:CompleteTheRaynaudFiber}]

We first consider the case in which $\Gamma$ is discrete or divisible. By \cite[Theorem 1.1]{combopaper} there is a $\Gamma$-rational completion $\completeCpx$ of $\smallcpx$ such that $\{\rec P\mid P\in\completeCpx\}=\smallfan$ and $\completeCpx$ is locally finite in $N_{\R}(\smallfan)$. 
Define $\halfspaceCompleteFan:=\{c(P)\mid P\in\completeCpx\}\cup\{\sigma\times\{0\}\mid\sigma\in\smallfan\}$. Then $\nbar{\Delta}$ is a fan by \cite[Lemma 4.6]{combopaper}. This Lemma applies because for each cone $c(P)$ we have that $c(P)\cap (N_{\R}\times \{0 \})=\rec P$, where $\rec P$ denotes the recession cone of $P$ as in Subsection 2.1 of \cite{combopaper}. Furthermore $\nbar{\Delta}$ is $\Gamma$-admissible and satisfies $\htzero{\halfspaceCompleteFan}{N_{\R}}=\smallfan$ and $\htone{\halfspaceCompleteFan}{N_{\R}}=\completeCpx$.
Letting $\nbar{\calY}$ be the normal $\T$-toric scheme corresponding to $\halfspaceCompleteFan$, because $\Gamma$ is discrete or divisible we have that $\nbar{\calY}$ is locally of finite type over $K^\circ$. 
Since $\halfspaceSmallFan$ is a subfan of $\halfspaceCompleteFan$, there is a $\T$-equivariant open immersion $\calY\into\nbar{\calY}$, and the induced map on generic fibers is an isomorphism because $\htzero{\halfspaceCompleteFan}{N_{\R}}=\htzero{\halfspaceSmallFan}{N_{\R}}$. 
Letting $Y=Y(\smallfan)$ be the common generic fiber of $\calY$ and $\nbar{\calY}$, and letting $\nbar{\frakY}$ be the formal completion of $\nbar{\calY}$ along its special fiber, Corollary \ref{coro:LocFinImpliesRaynaudInBerkovich} tells us that $\iota_{\nbar{\calY}}\fncolon\ray{\nbar{\frakY}}\to Y^{\an}$ identifies $\ray{\nbar{\frakY}}$ with $\trop^{-1}(\cl_{N_{\R}(\smallfan)}|\completeCpx|)=\trop^{-1}(\cl_{N_{\R}(\smallfan)}N_{\R})=\trop^{-1}(N_{\R}(\smallfan))=Y^{\an}$.

Now suppose that $\Gamma$ is neither discrete nor divisible.
By \cite[Theorem 1.1]{GammaAdmissibleCompletions} there is a finite separable totally ramified extension $L/K$ of valued fields such that $\calY_{L^\circ}$ admits a normal $\T_{L^\circ}$-equivariant completion. So by making the base-change to $L^\circ$ we may assume without loss of generality that $\calY$ admits a normal $\T$-equivariant completion, i.e., a $\T$-equivariant open immersion $\calY\into\calY'$ with $\calY'$ a normal $\T$-toric variety which is proper over $K^\circ$. Let $\halfspaceBigFan$ be the finite $\Gamma$-admissible fan in $N_{\R}\times\R_{\geq0}$ corresponding to $\calY'$ and let $\bigcpx:=\htone{\halfspaceBigFan}{N_{\R}}$. Because $\calY'$ is proper over $K^\circ$, \cite[Proposition 11.8]{GublerGuideTrop} tells us that $\halfspaceBigFan$ is complete, so $\bigcpx$ is also complete. As $\Gamma$ is not discrete and $\calY'$ is of finite type over $K^\circ$, $\bigcpx$ is a finite completion of $\smallcpx$ whose vertices are all in $N_{\Gamma}$. Thus, Theorem \cite[Theorem 1.1]{combopaper} tells us that there is a $\Gamma$-rational completion $\completeCpx$ of $\smallcpx$ such that $\{\rec P\mid P\in\completeCpx\}=\smallfan$, $\completeCpx$ is locally finite in $N_{\R}(\smallfan)$, and all of the vertices of $\completeCpx$ are in $N_{\Gamma}$. The remainder of the proof is exactly as in the previous case, with the exception that the justification of the fact that $\nbar{\calY}$ is locally of finite type in this case is that all of the vertices of $\completeCpx$ are in $N_{\Gamma}$.
\end{proof}

Finally, we can prove Theorem \ref{thm:ExistenceOfAlgebraizableFormalModels}.

\begin{proof}[Proof of Theorem \ref{thm:ExistenceOfAlgebraizableFormalModels}]
Let $X$ be a closed subscheme of a normal toric variety $Y$ over $K$.  Letting $T$ be the torus acting on $Y$, $M$ the character lattice of $T$, and $N$ the cocharacter lattice of $T$, we have that $Y=Y(\smallfan)$ for a finite rational fan $\smallfan$ in $N_{\R}$. If we let $\T:=\Spec K^\circ[M]$ then we can also view $Y$ as the $\T$-toric variety associated to $\halfspaceSmallFan:=\{\sigma\times\{0\}\mid\sigma\in\smallfan\}$. 

By \cite[Page 18]{OdaConvexBodiesAndAG} there is a finite rational completion $\bigfan$ of $\smallfan$. %Letting $\halfspaceBigFan:=\{c(\sigma)\mid \sigma\in\bigfan\}\cup\{\sigma\times\{0\}\mid\sigma\in\bigfan\}$, we have that $\calY(\halfspaceBigFan)$ is a normal equivariant completion of $Y$, viewed as a $\T$-toric variety. 
%%%%%%Awk line break
Considering the fan $\halfspaceBigFan:=\{c(\sigma)\mid \sigma\in\bigfan\}\cup\{\sigma\times\{0\}\mid\sigma\in\bigfan\}$, we have that $\calY(\halfspaceBigFan)$ is a normal equivariant completion of $Y$, viewed as a $\T$-toric variety. 
So by Theorem \ref{thm:CompleteTheRaynaudFiber} there is a normal $\T$-toric scheme $\nbar{\calY}$ locally of finite type over $K^\circ$ and a $\T$-equivariant open immersion $Y\into\nbar{\calY}$ identifying $Y$ with the generic fiber of $\nbar{\calY}$. Furthermore we have that for any closed subscheme $\calX\subset\nbar{\calY}$ which is flat over $K^\circ$ with generic fiber $\calX_K$ and formal completion $\frakX$ along the special fiber, $\iota_{\calX}\fncolon\ray{\frakX}\to(\calX_K)^{\an}$ is an isomorphism.

Let $\calX$ be the closure of $X$ in $\nbar{\calY}$, i.e., the scheme-theoretic image of the inclusion morphism $X\into\nbar{\calY}$. Then $\calX$ is a closed subscheme of $\nbar{\calY}$ which is flat over $K^\circ$ and has generic fiber $X$ \cite[Remark 4.6]{GublerGuideTrop}. So $\iota_{\calX}\fncolon\ray{\frakX}\to X^{\an}$ is an isomorphism. Thus $\frakX$ is an algebraizable formal model of $X^{\an}$.
\end{proof}

\bibliographystyle{alpha}
\bibliography{FormalModels}

\begin{thebibliography}{KKMSD73}

\bibitem[Ber90]{BerkovichBook}
Vladimir Berkovich.
\newblock {\em Spectral theory and analytic geometry over {non-Archimedean}
  fields}.
\newblock Number~33 in Mathematical Surveys and Monographs. Amer.\ Math.\ Soc.,
  Providence, RI, 1990.

\bibitem[Ber93]{Berkovich93Etale}
Vladimir Berkovich.
\newblock \'{E}tale cohomology for non-{A}rchimedean analytic spaces.
\newblock {\em Inst. Hautes \'{E}tudes Sci. Publ. Math.}, (78):5--161, 1993.

\bibitem[Ber94]{BerkovichFormalSchemesI}
Vladimir Berkovich.
\newblock Vanishing cycles for formal schemes.
\newblock {\em Invent. Math.}, 115:539--571, 1994.

\bibitem[BGS11]{RecessionFan}
Jos\'{e}~Ignacio Burgos~Gil and Mart\'{i}n Sombra.
\newblock When do the recession cones of a polyhedral complex form a fan?
\newblock {\em Discrete Comput.\ Geom.}, 46:789--798, 2011.

\bibitem[Bos14]{BoschBook}
Siegfried Bosch.
\newblock {\em Lectures on Formal and Rigid Geometry}, volume 2105 of {\em
  Lecture Notes in Mathematics}.
\newblock Springer International Publishing, 2014.

\bibitem[Bou95]{BourbakiGenTopI}
Nicolas Bourbaki.
\newblock {\em General Topology: Chapters 1-4}.
\newblock Springer-Verlag Berlin Heidelberg, 1995.

\bibitem[CF23]{combopaper}
Desmond Coles and Netanel Friedenberg.
\newblock Locally finite completions of polyhedral complexes.
\newblock Preprint, arXiv:2303.12334, 2023.

\bibitem[FGP14]{LimitsOfTrop}
Tyler Foster, Philipp Gross, and Sam Payne.
\newblock Limits of tropicalizations.
\newblock {\em Israel J. Math.}, 201(2):835--846, 2014.

\bibitem[Fri19]{GammaAdmissibleCompletions}
Netanel Friedenberg.
\newblock Normal completions of toric varieties over rank one valuation rings
  and completions of {$\Gamma$-admissible} fans.
\newblock Preprint, arXiv:1908:00064, 2019.

\bibitem[Gro65]{EGAIV2}
Alexander Grothendieck.
\newblock {\'{E}l\'{e}ments de G\'{e}om\'{e}trie Alg\'{e}brique (r\'{e}dig\'{e}
  avec la collaboration de Jean Dieudonn\'{e}): $\text{IV}_2$. \'{E}tude locale
  des sch\'{e}mas et des morphismes de sch\'{e}mas}.
\newblock {\em Inst. Hautes \'{E}tudes Sci. Publ. Math.}, 24:5--231, 1965.

\bibitem[GRW17]{TropicalSkeletons}
Walter Gubler, Joseph Rabinoff, and Annette Werner.
\newblock Tropical skeletons.
\newblock {\em Ann.\ Inst.\ Fourier (Grenoble)}, 67(5):1905--1961, 2017.

\bibitem[GS15]{GublerSoto}
Walter Gubler and Alejandro Soto.
\newblock Classification of normal toric varieties over a valuation ring of
  rank one.
\newblock {\em Doc.\ Math.}, 20:171--198, 2015.

\bibitem[Gub13]{GublerGuideTrop}
Walter Gubler.
\newblock A guide to tropicalizations.
\newblock In {\em Algebraic and combinatorial aspects of tropical geometry},
  volume 589 of {\em Contemp.\ Math.}, pages 125--189. Amer.\ Math.\ Soc.,
  Providence, RI, 2013.

\bibitem[KKMSD73]{ToroidalEmbeddings}
George Kempf, Finn Knudsen, David Mumford, and Bernard Saint-Donat.
\newblock {\em Toroidal Embeddings I}, volume 339 of {\em Lecture Notes in
  Mathematics}.
\newblock Springer-Verlag, Berlin, 1973.

\bibitem[MFK94]{GIT}
David Mumford, John Fogarty, and Frances Kirwan.
\newblock {\em Geometric Invariant Theory}, volume~34 of {\em Ergeb. Math.
  Grenzgeb.}
\newblock Springer-Verlag Berlin Heidelberg, third edition, 1994.

\bibitem[Oda78]{OdaTorusEmbeddings}
Tadao Oda.
\newblock {\em Lectures on Torus Embeddings and Applications (Based on joint
  work with Katsuya Miyake)}.
\newblock Number~58 in Tata Inst.\ Fund.\ Research, Bombay. Springer-Verlag,
  Berlin-Heidelberg-New York, 1978.

\bibitem[Oda88]{OdaConvexBodiesAndAG}
Tadao Oda.
\newblock {\em Covex Bodies and Algebraic Geometry: An Introduction to the
  Theory of Toric Varieties}, volume~15 of {\em Ergeb. Math. Grenzgeb.}
\newblock Springer-Verlag, Berlin Heidelberg, 1988.
\newblock Translated from the Japanese.

\bibitem[Pay09]{LimitOfTrops}
Sam Payne.
\newblock Analytification is the limit of all tropicalizations.
\newblock {\em Math.\ Res.\ Lett.}, 16(3):543--556, 2009.

\bibitem[Rab12]{RabinoffTropicalAnalyticGeometry}
Joseph Rabinoff.
\newblock Tropical analytic geometry, newton polygons, and tropical
  intersections.
\newblock {\em Adv.\ Math.}, 229:3192--3255, 2012.

\bibitem[Sum75]{SumihiroCompletionII}
Hideyasu Sumihiro.
\newblock Equivariant completion {II}.
\newblock {\em J.\ Math.\ Kyoto Univ.}, 15(3):573--605, 1975.

\bibitem[Tem15]{TemkinIntroToBerkovichSpaces}
Michael Temkin.
\newblock Introduction to {Berkovich} analytic spaces.
\newblock In Antoine Ducros, Charles Favre, and Johannes Nicaise, editors, {\em
  Berkovich Spaces and Applications}, volume 2119 of {\em Lecture Notes in
  Mathematics}, pages 3--66. Springer International Publishing, 2015.

\bibitem[W\l93]{WlodarczykEmbedding}
Jaros\l{}aw W\l{}odarczyk.
\newblock Embeddings in toric varieties and prevarieties.
\newblock {\em J.\ Alg.\ Geom.}, 2(4):705--726, 1993.

\end{thebibliography}

%\printbibliography

\end{document}